\documentclass[12pt,a4paper]{amsart}
\usepackage[utf8]{inputenc}
\usepackage[active]{srcltx}
\usepackage{amsthm}

\makeatletter

\pdfpageheight\paperheight
\pdfpagewidth\paperwidth

\numberwithin{equation}{section}
\numberwithin{figure}{section}

\usepackage{amsmath}

\usepackage{amssymb}
\usepackage{stmaryrd}
\usepackage{amsthm}
\usepackage[mathscr]{eucal}
\usepackage{amscd}
\usepackage[all]{xy}
\usepackage{url}

\newtheorem{Thm}{Theorem}[subsection]

\usepackage{enumerate}

\usepackage{comment}

\usepackage{hyperref}
\hypersetup{%
pdftitle={Preprint},
pdfauthor={},
pdfkeywords={},
bookmarksnumbered,
pdfstartview={FitH},
breaklinks=true,
colorlinks=true,
}%
\usepackage[all]{hypcap} 

\usepackage{breakurl}

\usepackage{color}

\newtheorem{Lem}[Thm]{Lemma}

\newtheorem{Conj}[Thm]{Conjecture}

\newtheorem{Eg}[Thm]{Example}
\newtheorem{Rem}[Thm]{Remark}
\newtheorem{Def}[Thm]{Definition}

\newtheorem*{Def*}{Definition}
\newtheorem*{Thm*}{Theorem}

\newtheorem*{Conj*}{Conjecture}

\newcommand{\Z}{\mathbb{Z}}
\newcommand{\N}{\mathbb{N}}

\renewcommand{\hat}[1]{\widehat{#1}}
\renewcommand{\tilde}[1]{\widetilde{#1}}

\newcommand{\opname}[1]{\operatorname{\mathsf{#1}}}

\newcommand{\pr}{\opname{pr}}

\newcommand{\inj}{\opname{inj}}

\renewcommand{\deg}{\opname{deg}}

\newcommand{\Hf}{{\frac{1}{2}}}

\newcommand{\Rm}[1]{{\longmapsto}}
\newcommand{\Lm}[1]{{\longmapsfrom}}

\newcommand{\cA}{{\mathcal A}}

\newcommand{\cT}{{\mathcal T}}

\newcommand{\bI}{{\mathbf I}}

\newcommand{\bL}{{\mathbf L}}

\newcommand{\Bm}{{\mathbf m}}

\newcommand{\tB}{{\widetilde{B}}}

\newcommand{\can}{L}
\newcommand{\gen}{\mathbb{L}}

\newcommand{\qClAlg}{{\cA}}

\usepackage{tikz}
\usetikzlibrary{positioning,shapes,shadows,arrows,snakes}

\pgfdeclarelayer{edgelayer}
\pgfdeclarelayer{nodelayer}
\pgfsetlayers{edgelayer,nodelayer,main}

\tikzstyle{none}=[inner sep=0pt]
\tikzstyle{black box}=[draw=black, fill=black!25]
\tikzstyle{white box}=[draw=black, fill=white]
\tikzstyle{black circle}=[circle,draw=black!50, fill=black!25]
\tikzstyle{red circle}=[circle,draw=red!50, fill=red!25]
\tikzstyle{blue circle}=[circle,draw=blue!50, fill=blue!25]
\tikzstyle{green circle}=[circle,draw=green!50, fill=green!25]
\tikzstyle{yellow circle}=[circle,draw=yellow!50, fill=yellow!25]

\newcommand{\thistheoremname}{}
\newtheorem*{genericthm*}{\thistheoremname}
\newenvironment{namedthm*}[1]
  {\renewcommand{\thistheoremname}{#1}%
   \begin{genericthm*}}
  {\end{genericthm*}}

\renewcommand{\inj}{{\bI}}
\renewcommand{\can}{{\bL}}

\newcommand{\degL}{\mathrm{D}}

\makeatother

\begin{document}

\title[]{Compare triangular bases of acyclic quantum cluster algebras}
\author{Fan Qin}

\email{qin.fan.math@gmail.com}
\begin{abstract}
Given a quantum cluster algebra, we show that its triangular bases
defined by Berenstein and Zelevinsky and those defined by the author
are the same for the seeds associated with acyclic quivers. This result
implies that the Berenstein-Zelevinsky's basis contains all the quantum
cluster monomials.

We also give an easy proof that the two bases are the same for the
seeds associated with bipartite skew-symmetrizable matrices.
\end{abstract}

\maketitle
\tableofcontents{}

\section{Introduction}

\label{sec:intro}

\subsection{Cluster algebras}

In \cite{fomin2002cluster}, Fomin and Zelevinsky invented cluster
algebras as a combinatorial approach to dual canonical bases of quantum
groups (discovered by Lusztig \cite{Lusztig90} and Kashiwara \cite{Kashiwara90}
independently). The quantum cluster algebras were later introduced
in \cite{BerensteinZelevinsky05}. These algebras possess many seeds,
which are constructed recursively by an algorithm called mutation.
Every seed consists of some skew-symmetrizable matrix and a collection
of generators called (quantum) cluster variables. We might view these
seeds as analog of local charts of algebraic varieties\footnote{In fact, we have a family of varieties called cluster varieties, whose
local charts are tori, local coordinate functions are cluster variables,
and transition maps are determined by the matrices in the seeds, cf.
\cite{FockGoncharov03}.}.

There are many attempts to ``good'' bases of cluster algebras, cf.
\cite{GeissLeclercSchroeer10,GeissLeclercSchroeer10b,GeissLeclercSchroeer11}
\cite{musiker2013bases,thurston2014positive} \cite{HernandezLeclerc09}
\cite{Nakajima09,KimuraQin14,Qin12} \cite{lee2014greedy,lee2014greedyPNAS}
\cite{gross2014canonical} \cite{Qin15} \cite{KKKO15}. In view of
the original motivation of Fomin and Zelevinsky, a good basis should
contain all the quantum cluster monomials (monomials of quantum cluster
variables belonging to the same seed).

\subsection{Berenstein-Zelevinsky's triangular basis approach}

In \cite{BerensteinZelevinsky2012}, Berenstein and Zelevinsky proposed
the following new approach to good bases of quantum cluster algebras:

\begin{itemize}

\item Inspired by the Kazhdan-Lusztig theory, construct a triangular
basis $C^{t}$ in each seed $t$ such that it contains all the quantum
cluster monomials in that seed. More precisely, first construct a
basis consisting of some ordered products of quantum cluster variables,
then Lusztig's lemma \cite[Theorem 1.1]{BerensteinZelevinsky2012}
guarantees a unique new basis whose transition matrix from the old
one is unitriangular, whence the name triangular basis.

\item Prove that these triangular bases give rise to a common basis
for all seeds.

\end{itemize}

If this approach works, then we have a common triangular basis containing
the quantum cluster monomials in all seeds. However, Berenstein-Zelevinsky's
construction only works for those special seeds of acyclic type, cf.
Section \ref{sub:BZ-basis} for the definition. They arrived at a
common basis for the acyclic seeds, which we call the $BZ$-basis
and denote by $C$.

On the other hand, it is known that the quantum cluster algebras associated
with acyclic quiver and $z$-coefficient pattern are isomorphic to
some quantum unipotent subgroups and, consequently, inherit the dual
canonical bases, cf. \cite{GeissLeclercSchroeer11}\cite{KimuraQin14}.
In \cite{KimuraQin14}, Kimura and the author showed that, for such
quantum cluster algebras, the dual canonical bases contain all the
quantum cluster monomials. It is natural to propose the following
conjecture.

\begin{Conj}\label{conj:BZ_basis_good}

For quantum cluster algebras associated with an acyclic quiver and
$z$-coefficient pattern, its dual canonical basis agrees with Berenstein-Zelevinsky's
triangular basis $C$.

\end{Conj}

The verification of this conjecture would imply the desired property
that Berenstein-Zelevinsky's triangular basis contains all quantum
cluster monomials.

\subsection{Different triangular bases in monoidal categorification}

Inspired by this new approach of Berenstein-Zelevinsky, in \cite{Qin15},
in order to prove monoidal categorification conjectures of quantum
cluster algebras, the author introduced very different triangular
bases for injective-reachable quantum cluster algebras. For every
seeds $t$, we can define a such triangular bases $\can^{t}$, cf.
Section \ref{sub:Triangular-basis}. 

There are two crucial differences of the common triangular basis $\can$
in \cite{Qin15} with the basis $C$ of Berenstein-Zelevinsky:

\begin{enumerate}

\item The basis is unique but its existence cannot be guaranteed,
because Lusztig's lemma does not apply.

\item The expectation from Fock-Goncharov basis conjecture is included
in the definition and plays an important role.

\end{enumerate}

\subsection{Results}

We have two very different constructions of triangular bases. It is
desirable to compare these bases, which are both defined for acyclic
seeds. The main result of this paper claims that they are the same
for quantum cluster algebras arising from acyclic skew-symmetric matrices
(or, equivalently, from acyclic quivers).

\begin{Thm}[Main result]\label{thm:acyclic}

Let $\qClAlg$ be a quantum cluster algebras who has a seed $t$ with
an acyclic skew-symmetric matrix $B(t)$. Then in this seed, its triangular
basis $\can^{t}$ in \cite{Qin15} agrees with Berenstein-Zelevinsky's
triangular basis $C$.

\end{Thm}

Notice that, for the quantum cluster algebra arising from an acyclic
quiver and $z$-coefficient pattern, its common triangular bases in
\cite{Qin15} is the dual canonical basis. Therefore, our main result
Theorem \ref{thm:acyclic} implies Conjecture \ref{conj:BZ_basis_good}.

Our proof is based on ideas and techniques developed by the author
in \cite{Qin15}, in particular, the maximal degree tracking and the
composition of unitriangular transitions. The triangular bases treated
in this paper are much easier than those in \cite{Qin15} and our
paper does not depend on the long proof there. In particular, we give
a self-contained proof that the triangular bases $\can^{t}$ in different
acyclic seeds $t$ are the same, cf. Theorem \ref{thm:common_triangular_basis}.

We could further propose the following natural conjecture.

\begin{Conj}\label{conj:symmetrizable}

The triangular basis $\can^{t}$ agrees with Berenstein-Zelevinsky's
triangular basis $C$ in seeds associated with acyclic skew-symmetrizable
seeds.

\end{Conj}

In a previous private communication with Zelevinsky, the author pointed
out that for bipartite orientation, this conjecture is true. The details
will be given in the appendix, cf. Theorem \ref{thm:bipartite}.

\section*{Acknowledgments}

The author thanks Andrei Zelevinsky and Kyungyong Lee for conversations
on acyclic cluster algebras. He thanks Yoshiyuki Kimura, Qiaoling
Wei and Changjian Fu for remarks.

\section{Preliminaries}

\subsection{Quantum cluster algebras}

We recall the definition of quantum cluster algebras by \cite{BerensteinZelevinsky05}
and follow the convention in \cite{Qin15}. Let $[x]_{+}$ denote
$\max(x,0)$. Let $\tilde{B}$ be an $m\times n$ integer matrix with
$n\leq m$. Its $n\times n$ upper submatrix $B$ is called the principal
part. Assume that $\tilde{B}$ is of rank $n$ and $B$ skew-symmetrizable
(namely, there exists a diagonal matrix with strictly positive integer
diagonal entries such that its product with $B$ is skew-symmetric).
We can choose $\Lambda$ an $m\times m$ skew-symmetric integer matrix
such that $\tilde{B}^{T}\Lambda=(\begin{array}{cc}
D & 0\end{array})$ for some diagonal matrix $D$ with strictly positive integer diagonal
entries. Such a pair $(\tB,\Lambda)$ is called a compatible pair.

A quantum seed $t$ (or seed for simplicity) consists of a compatible
pair $(\tB(t),\Lambda(t))$ and a collection of indeterminate $X_{i}(t)$,
$1\leq i\leq m$, called $X$-variables. Let $\{e_{i}\}$ denote the
natural basis of $\Z^{m}$ and $X(t)^{e_{i}}=X_{i}(t)$. We define
the corresponding quantum torus $\cT(t)$ to be the Laurent polynomial
ring $\Z[q^{\pm\Hf}][X(t)^{g}]_{g\in\Z^{m}}$ with the usual addition
$+$, the usual multiplication $\cdot$, and the twisted product

\begin{align*}
X(t)^{g}*X(t)^{h} & =q^{\Hf\Lambda(t)(g,h)}X(t)^{g+h},
\end{align*}

where $\Lambda(t)(\ ,\ )$ denote the bilinear form on $\Z^{m}$ such
that 
\begin{align*}
\Lambda(t)(e_{i},e_{j}) & =\Lambda(t)_{ij}.
\end{align*}
$\cT(t)$ admits a bar-involution $\overline{(\ )}$ which is $\Z$-linear
such that 
\begin{align*}
\overline{q^{s}X(t)^{g}} & =q^{-s}X(t)^{g}.
\end{align*}

Notice that all Laurent monomials in $\cT(t)$ commute with each other
up to a $q$-power, which is called $q$-commute.

Let $b_{ij}$ denote the $(i,j)$-entry of $\tB(t)$. We define the
$Y$-variables to be the following Laurent monomials:

\begin{align*}
Y_{k}(t) & =X(t)^{\sum_{1\leq i\leq m}[b_{ik}]_{+}e_{i}-\sum_{1\leq j\leq m}[-b_{jk}]_{+}e_{j}}.
\end{align*}

For any direction $1\leq k\leq n$, the following operation (called
the mutation $\mu_{k}$) gives us a new seed $t'=\mu_{k}t=((X_{i}(t'))_{1\leq i\leq m},\tB(t'),\Lambda(t'))$:

\begin{itemize}

\item $X_{i}(t')=X_{i}(t)$ if $i\neq k$,

\item $X_{k}(t')=X(t)^{-e_{k}+\sum_{i}[b_{ik}]_{+}e_{i}}+X(t)^{-e_{k}+\sum_{j}[-b_{jk}]_{+}e_{j}}$,

\item $\tB(t')=(b_{ij}')$ is determined by $\tB(t)=(b_{ij})$:

$\begin{cases}
b'_{ik} & =-b_{ki}\\
b_{ij}' & =b_{ij}+[b_{ik}]_{+}[b_{kj}]_{+}-[-b_{ik}]_{+}[-b_{kj}]_{+}\qquad\mathrm{if}\ i,j\neq k
\end{cases}$

\item $\Lambda(t')$ is skew-symmetric and satisfies 
\begin{align*}
\begin{cases}
\Lambda(t')_{ij} & =\Lambda(t)_{ij}\qquad i,j\neq k\\
\Lambda(t')_{ik} & =\Lambda(t)(e_{i},-e_{k}+\sum_{j}[-b_{jk}]_{+}e_{j})\qquad i\neq k
\end{cases}
\end{align*}

\end{itemize}

The quantum torus $\cT(t')$ for the new seed $t'$ is defined similarly.
Notice that, by \cite[Proposition 6.2]{BerensteinZelevinsky05}, any
$Z\in\cT(t)\cap\cT(t')$ is bar-invariant in $\cT(t)$ if and only
if it is bar-invariant in $\cT(t')$.

We define a quantum cluster algebra $\qClAlg$ as the following:

\begin{itemize}

\item Choose an initial seed $t_{0}=((X_{1},\cdots,X_{m}),\tB,\Lambda)$.

\item All the seeds $t$ are obtained from $t_{0}$ by iterated mutations
at directions $1\leq k\leq n$.

\item $\qClAlg=\Z[q^{\pm\Hf}][X_{n+1}^{-1},\cdots,X_{m}^{-1}][X_{i}(t)]_{t,1\leq i\leq m}.$

\end{itemize}

The $X$-variables $X_{i}(t)$ in the seeds are called the quantum
cluster variables. We call $X_{n+1},\ldots,X_{m}$ the frozen variables
or the coefficients.

The correction technique developed in \cite[Section 9]{Qin12} provides
a convenient tool for studying the bases of $\qClAlg$, cf. \cite[Section 5]{Qin15}
for a summary. It tells us that most phenomenons and properties of
bases keep unchanged when we change the coefficient part of the seed
$t$, namely the lower $(m-n)\times n$ submatrix $B^{c}(t)$ of $\tB(t)$,
or when we change $\Lambda(t)$.

Finally, notice that to each rank $n$ quiver $Q$, we can associate
an $n\times n$ skew-symmetric matrix $B$ such that its entry $b_{ij}$
is given by the difference of the number of arrows from $i$ to $j$
with that of $j$ to $i$. All skew-symmetric matrices arise in this
way. So, if the matrix $B(t)$ of a seed $t$ is skew-symmetric, we
say $t$ is skew-symmetric or $t$ arises from a quiver; if $B(t)$
is skew-symmetrizable, we say $t$ is skew-symmetrizable.

\subsection{Triangular basis\label{sub:Triangular-basis}}

Choose any seed $t$. We recall the following notions introduced in
\cite[Section 3.1]{Qin15}

\begin{Def}[Pointed elements and normalization]A Laurent polynomial
$Z$ in the quantum torus $\cT(t)$ is said to be pointed if it takes
the form

\begin{align}
Z & =X(t)^{g}\cdot(1+\sum_{0\neq v\in\N^{n}}c_{v}Y(t)^{v}),\label{eq:pointed}
\end{align}

for some coefficients $c_{v}\in\Z[q^{\pm\Hf}]$. 

In this case, $Z$ is said to be pointed at degree $g$, and we denote
$\deg^{t}Z=g$.

If $Z=q^{s}X(t)^{g}(1+\sum_{0\neq v\in\N^{n}}c_{v}Y(t)^{v})$ for
some $s\in\frac{\Z}{2}$, we use $[Z]^{t}$ to denote the pointed
element $q^{-s}Z$ and call it the normalization of $Z$ in $\cT(t)$.

\end{Def}

Notice that all the quantum cluster variables are pointed.

In order to say that a pointed element has a unique maximal degree,
we need to introduce the following partial order.

\begin{Def}[Degree lattice and dominance order]

We call $\Z^{m}$ the degree lattice and denote it by $\degL(t)$.
Its dominance order $\prec_{t}$ is defined to be the partial order
such that $g'\prec_{t}g$ if and only if $g'=g+\deg^{t}Y(t)^{v}$
for some $0\neq v\in\N^{n}$.

\end{Def}

We might omit the symbol $t$ in $X_{i}(t)$, $I_{k}(t)$,$\prec_{t}$,
$\deg^{t}$ or $[\ ]^{t}$ for simplicity.

\begin{Lem}[{\cite{Qin15}[Lemma 3.1.2]}]\label{lem:finite_interval}

For any $g'\preceq_{t}g$ in $\Z^{m}$, there exists finitely many
$g''\in\Z^{m}$ such that $g'\preceq_{t}g''\preceq_{t}g$.

\end{Lem}

Assume that, in $\cT(t)$, we have (possibly infinitely many) elements
$\gen_{j}$ pointed in different degrees. Let we denote $\gen_{j}=\sum_{g\in\Z^{m}}c_{g;j}X^{g}$
where $c_{g;j}\in\Z[q^{\pm\Hf}]$. A linear combination $\sum_{j}a_{j}\gen_{j}$
with $a_{j}\in\Z[q^{\pm\Hf}]$ is well defined and contained in $\cT(t)$
if $\sum_{j}a_{j}c_{g;j}$ is a finite sum for all $g\in\Z^{m}$ and
vanishes except for finitely many $g$.

Assume that $Z$ be a Laurent polynomial in $\cT(t)$ such that it
is a well defined linear combination of $\gen_{j}$:

\begin{align}
Z & =\sum_{j}a_{j}\gen_{j},\qquad a_{j}\in\Z[q^{\pm\Hf}].\label{eq:decomposition}
\end{align}

We say that this decomposition $\prec_{t}$-triangular if there exists
a unique $\prec_{t}$-maximal element $\deg^{t}\gen_{0}$ in $\{\deg^{t}\gen_{j}\}$.
It is further called $\prec_{t}$-unitriangular if $a_{0}=1$, or
$(\prec_{t},\Bm)$-triangular if $a_{j}\in\Bm=q^{-\Hf}\Z[q^{-\Hf}]$
for $j\neq0$. A set $\{Z\}$ is said to be $(\prec_{t},\Bm)$-unitriangular
to $\{\gen_{j}\}$ if all its elements $Z$ has such property.

\begin{Lem}[{\cite{Qin15}[Lemma 3.1.9]}]

If the decomposition \eqref{eq:decomposition} is $\prec_{t}$-triangular,
then it is the unique $\prec_{t}$-triangular decomposition of $Z$
in $\{\gen_{j}\}$.

\end{Lem}

\begin{proof}

Thanks to Lemma \ref{lem:finite_interval}, we can recursively determine
all the coefficients $a_{j}$ of $\gen_{j}$ in \eqref{eq:decomposition},
starting from the higher $\prec_{t}$-order Laurent degrees, cf. \cite{Qin15}[Remark 3.1.8].

\end{proof}

The following lemma will be useful. It allows us to switch to the
desired dominance order.

\begin{Lem}[{\cite{Qin15}[Lemma 3.1.9]}]\label{lem:has_triangular_order}

(i) If \eqref{eq:decomposition} is a finite decomposition of a pointed
element $Z$, then it is $\prec_{t}$-unitriangular.

(ii) If, further, all but one coefficients in \eqref{eq:decomposition}
belong to $\Bm$, then \eqref{eq:decomposition} is $(\prec_{t},\Bm)$-unitriangular.

\end{Lem}

\begin{proof}

(i) We recall the proof in \cite{Qin15}[Lemma 3.1.9]. Compare maximal
degrees of both hand sides of a finite decomposition, we obtain that
the finite set $\{\deg\gen_{j}\}$ contains a unique maximal element
$\deg\gen_{0}$ for some $\gen_{0}$ such that $\deg\gen_{0}=\deg Z$.
So this decomposition is $\prec_{t}$-triangular. Finally, $a_{0}=1$
because $Z$ has coefficient $1$ in its leading degree.

(ii) By (i), $Z$ admits a $\prec_{t}$-unitriangular decomposition.
The hypothesis in (ii) simply tells us that the coefficients other
than the leading coefficient (equals 1) belong to $\Bm$.

\end{proof}

For any $1\leq k\leq n$, let $I_{k}(t)$ denote\footnote{We use the notation $I_{k}$ because this cluster variable corresponds
to the $k$-th indecomposable injective module of a quiver with potential
\cite{DerksenWeymanZelevinsky08,DerksenWeymanZelevinsky09}.} the unique quantum cluster variable (if it exists) such that $\mathrm{pr}_{n}\deg^{t}I_{k}(t)=-e_{k}$,
where $\pr_{n}$ is the projection of $\Z^{m}$ onto the first $n$-components.
The quantum cluster algebra $\qClAlg$ is said to be \textit{injective
reachable} if $I_{k}(t)$ exists for any $1\leq k\leq n$. This property
is independent of the choice of the seed $t$ by \cite{Plamondon10a}\cite{gross2014canonical}.
In this case, the quantum cluster variables $I_{k}(t)$, $1\leq k\leq n$,
$q$-commute with each other because they belong to the same seed
(denoted by $t[1]$ in \cite{Qin15}).

\begin{Rem}\label{rem:acyclic_injectives}

In the convention of Section \ref{sub:BZ-basis}, if $B(t)$ is acyclic,
we can obtain the quantum cluster variables $I_{k}$, $\forall1\leq k\leq n$,
by applying the sequence of mutations on each vertex $1,\cdots,n$
such that the their order increases with respect to $\triangleleft$.
In particular, the corresponding cluster algebra is injective reachable.
See Example \ref{eg:Kronecker} for an explicit calculation.

\end{Rem}

\begin{Def}[Triangular basis {\cite[Definition 6.1.1]{Qin15}}]\label{def:triangular_basis}

The triangular basis $\can^{t}$ for the seed $t$ is defined to be
the basis of the quantum cluster algebra $\qClAlg$ such that

\begin{itemize}

\item The quantum cluster monomials $[\prod_{1\leq i\leq m}X_i(t)^{u_i}]^t$,$[\prod_{1\leq k\leq n}I_{k}(t)^{v_k}]^t$ belong to $\can^{t}$, $\forall u_i,v_k\in\N$.

\item (bar-invariance) The basis elements are invariant under the
bar involution in $\cT(t)$.

\item (parametrization) The basis elements are pointed, and we have
the bijection 
\begin{align*}
\deg^{t}:\can^{t} & \simeq\degL(t)=\Z^{m}.
\end{align*}

\item (triangularity) For any $X_{i}(t)$ and $S\in\can^{t}$, we
have\\
\begin{align*}
[X_{i}(t)*S]^{t} & =b+\sum c_{b'}\cdot b',
\end{align*}
where $\deg^{t}b'\prec_{t}\deg^{t}b=\deg^{t}X_{i}(t)+\deg^{t}S$ and
the coefficients $c_{b'}\in\Bm=q^{-\Hf}\Z[q^{-\Hf}]$.

\end{itemize}

\end{Def}

It is easy to show that if $\can^{t}$ exists, then it is unique by
the triangularity and bar-invariance, cf \cite[Lemma 6.2.6(i)]{Qin15}.
In order to study $\can^{t}$, \cite{Qin15} introduced the injective
pointed set $\inj^{t}$ in the seed $t$:

\begin{align*}
\inj^{t} & =\{\inj^{t}(f,u,v)|f\in\Z^{[n+1,m]},\ u,\ v\in\N^{[1,n]},u_{k}v_{k}=0\forall k\in[1,n]\}\\
\inj^{t}(f,u,v) & =[\prod_{n+1\leq i\leq m}X_{i}(t)^{f_{i}}*\prod_{1\leq k\leq m}X_{k}(t)^{u_{k}}*\prod_{1\leq k\leq m}I_{k}(t)^{v_{k}}]^{t}
\end{align*}

This is a linearly independent family of pointed elements contained
in $\qClAlg$. By the triangularity of $\can^{t}$, the set of pointed
elements $\inj^{t}$ is $(\prec_{t},\Bm)$-unitriangular to $\can^{t}$.
It follows that $\can^{t}$ is also $(\prec_{t},\Bm)$-unitriangular
to $\inj^{t}$, cf. \cite[Lemma(inverse transition)]{Qin15}.

\begin{Eg}[Type $A_3$]\label{eg:A_3_triangular}

Consider the matrix $\tB=\left(\begin{array}{ccc}
0 & -1 & 0\\
1 & 0 & -1\\
0 & 1 & 0\\
-1 & 1 & 0\\
0 & -1 & 1\\
0 & 0 & -1
\end{array}\right)$, which is the matrix of the ice quiver in Figure \ref{fig:A3Quiver}.

In the convention of \cite{KimuraQin14}, its principal part is an
acyclic type $A_{3}$ quiver and coefficient part the $z$-pattern.
There is a natural matrix $\Lambda$ such that $(\tB,\Lambda)$ is
compatible. The corresponding quantum cluster algebra $\qClAlg$ is
isomorphic to the quantum unipotent subgroup $A_{q}(\mathfrak{\mathfrak{n}}(c^{2}))$
localized at the coefficients $X_{4},X_{5},X_{6}$, where the Coxeter
word $c=s_{3}s_{2}s_{1}$ (read from right to left).

The quantum cluster variables $I_{1},I_{2},I_{3}$ are obtained from
consecutive mutations at $1,2,3$. Our pointed element $\inj(f,u,v)$ 

\begin{align*}
\inj(f,u,v) & =[X_{4}^{f_{4}}*X_{5}^{f_{5}}*X_{6}^{f_{6}}*X_{1}^{u_{1}}*X_{2}^{u_{2}}*X_{3}^{u_{3}}*I_{1}^{v_{1}}*I_{2}^{v_{2}}*I_{3}^{v_{3}}]
\end{align*}

is a localized dual PBW basis element (rescaled by a $q$-power),
and the triangular basis is the localized (rescaled) dual canonical
basis, cf. \cite{KimuraQin14}.

\end{Eg}

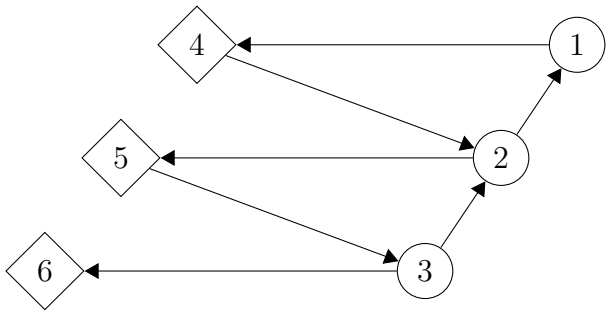
\begin{figure}[htb!]  \centering \beginpgfgraphicnamed{fig:A3Quiver}   \begin{tikzpicture}     \node [shape=circle, draw] (v1) at (1,-3) {3};      \node  [shape=circle, draw] (v2) at (2,-1.5) {2};      \node [shape=circle,  draw] (v3) at (3,0) {1};
\node [shape=diamond, draw] (v4) at (-4,-3) {6};      \node  [shape=diamond, draw] (v5) at (-3,-1.5) {5};      \node [shape=diamond,  draw] (v6) at (-2,0) {4};
    \draw[-triangle 60] (v1) edge (v2);      \draw[-triangle 60] (v2) edge (v3);         \draw[-triangle 60] (v5) edge (v1);          \draw[-triangle 60] (v6) edge (v2);           \draw[-triangle 60] (v1) edge (v4);      \draw[-triangle 60] (v2) edge (v5);      \draw[-triangle 60] (v3) edge (v6);            \end{tikzpicture} \endpgfgraphicnamed \caption{Acyclic A3 quiver with $z$-pattern} \label{fig:A3Quiver} \end{figure}

\begin{Lem}[{Substitution \cite{Qin15}[Lemma 6.4.4]}]\label{lem:substitution}

If a pointed element $Z$ is $(\prec_{t},\Bm)$-unitriangular to $\can^{t}$,
so does $[\prod_{n+1\leq i\leq m}X^{f_{i}}*X^{u}*Z*I^{v}]$ for any
$f\in\Z^{[n+1,m]}$,$u,v\in\N^{n}$.

\end{Lem}

\begin{proof}

$Z$ is $(\prec_{t},\Bm)$-unitriangular to $\inj^{t}$ and admits
a $(\prec_{t},\Bm)$-unitriangular decomposition

\begin{align*}
Z & =\sum_{s}a_{s}\inj^{t}(f^{(s)},u^{(s)},v{}^{(s)}).
\end{align*}

Replace $Z$ by this decomposition in $[\prod_{n+1\leq i\leq m}X^{f_{i}}*X^{u}*Z*I^{v}]$,
the result is $(\prec_{t},\Bm)$-unitriangular to $\can^{t}$, by
the triangularity of $\can^{t}$ and comparison of $q$-powers (cf.
\cite[Lemma 6.2.4]{Qin15}).

\end{proof}

\subsection{Berenstein-Zelevinsky's triangular basis\label{sub:BZ-basis}}

Work in some chosen seed $t$, whose symbol we often omit. Assume
that its principal part $B=B(t)$ is acyclic, namely, there exists
an order $\triangleleft$ on the vertex $\{1,\ldots,n\}$ such that
$b_{ij}\leq0$ whenever $i\triangleleft j$. In this case, $t$ is
called an acyclic seed. If $i\triangleleft j$, we say $i$ is $\triangleleft$-inferior
than $j$, and also denote $j\triangleright i$.

A vertex $j\in[1,n]$ is said to be a source point in $t$ if $j$
is $\triangleleft$-maximal, namely, $j\triangleright k$ for all
$1\leq k\leq n$. Similarly, it is called a sink point in $t$ if
$j$ is $\triangleleft$-minimal, namely, $j\triangleleft k$ for
all $1\leq k\leq n$.

For any $1\leq k\leq n$, let $b_{k}=\tilde{B}e_{k}$ denote the $k$-th
column of $\tilde{B}$. Let $S_{k}=S_{k}(t)$ denote\footnote{We use the symbol $S_{k}$ because this cluster variable corresponds
to the $k$-th simple $S_{k}$ in an associated quiver with potential.} the quantum cluster variable $X_{k}(\mu_{k}t)$. Notice that $S_{k}=X^{-e_{k}+[-b_{k}]_{+}}\cdot(1+Y_{k})$
and we have $\deg S_{k}=-e_{k}+[-b_{k}]_{+}$, where $[-b_{k}]_{+}$
denote $([-b_{jk}]_{+})_{1\leq j\le m}$.

For any $a\in\Z^{m}$, Bernstein and Zelevinsky defined the standard
monomials

$E_{a}=[\prod_{n<j\leq m}X^{a_{j}}*\prod_{1\leq k\leq n}X_{k}^{[a_{k}]_{+}}*\prod_{1\leq k\leq n}^{\triangleleft}S_{k}^{[-a_{k}]_{+}}]$,

where the last factor is the product with increasing $\triangleleft$
order, cf. \cite[(1.17) (1.22) Remark 1.3]{BerensteinZelevinsky2012}.

Define $r(a)=\sum_{1\leq k\leq n}[-a_{k}]_{+}$. Define partial order
$a\prec_{BZ}a'$ if and only if $r(a)<r(a')$.

\begin{Def}

The Berenstein-Zelevinsky's acyclic triangular basis for the seed
$t$ is defined to be the basis $C^{t}=\{C_{a}\}$ of $\qClAlg$ such
that each $C_{a}$ is bar-invariant and $(\prec_{BZ},q^{\Hf}\Z[q^{\Hf}])$-triangular
to the basis $\{E_{a}\}$. 

\end{Def}

We call $C^{t}$ the BZ-basis for simplicity. Applying the bar involution,
we obtain that $C^{t}$ is $(\prec_{BZ},\Bm)$-triangular to $\{\overline{E}_{a}\}$,
where 
\begin{align*}
\overline{E}_{a} & =[\prod_{1\leq k\leq n}^{\triangleright}S_{k}^{[-a_{k}]_{+}}*\prod_{1\leq k\leq n}X_{k}^{[a_{k}]_{+}}*\prod_{n<j\leq m}X^{a_{j}}]
\end{align*}

where the first factor is the product with decreasing $\triangleleft$
order.

\begin{Eg}

Let us continue Example \ref{eg:A_3_triangular}. The standard monomials,
after the bar involution, gives us 
\begin{align*}
\overline{E}_{a} & =[S_{3}^{[-a_{3}]_{+}}*S_{2}^{[-a_{2}]_{+}}*S_{1}^{[-a_{1}]_{+}}*X_{1}^{[a_{1}]_{+}}*X_{2}^{[a_{2}]_{+}}*X_{3}^{[a_{3}]_{+}}**X_{4}^{a_{4}}*X_{5}^{a_{5}}*X_{6}^{a_{6}}].
\end{align*}

Notice that $X_{4},X_{5},X_{6}$ $q$-commute with all the factors.

\end{Eg}

\begin{Thm}[{\cite{BerensteinZelevinsky12}[Theorem 1.4]}]

The Berenstein-Zelevinsky's triangular basis $C^{t}$ is independent
of the acyclic seed $t$ chosen, which we denote by $C$.

\end{Thm}

\section{Compare triangular bases}

\subsection{Basic results}

Let we choose and work with any seed $t$ whose matrix $B(t)$ is
acyclic.

\begin{Lem}\label{lem:change_BZ_order}

For any acyclic seed $t$, each $C_{a}$ is $(\prec_{t},\Bm)$-unitriangular
to $\{\overline{E}_{a}\}$.

\end{Lem}

\begin{proof}

Each $C_{a}$ is a finite linear combination of $\{\overline{E}_{a}\}$
with one term of coefficient $1$ and others of coefficients in $\Bm$.
This decomposition is $\prec_{t}$-triangular by Lemma \ref{lem:has_triangular_order}.

\end{proof}

\begin{Lem}\label{lem:keep_pointed}

If $n$ is a source point, then $\overline{E}_{a}$ remains pointed
in $t'=\mu_{n}t$.

\end{Lem}

\begin{proof}

It might be possible to deduce this result from the existence of common
Berenstein-Zelevinsky triangular bases in $t$ and $t'$. Let us give
an alternative elementary verification.

In order to show that the $q$-normalization factor producing by the
factors of $\overline{E}_{a}$ remains unchanged in $\cT(t')$, it
suffices to show that, for any $1\leq i,j\leq m$, $1\leq l<k\leq n$,
$i\neq k$, we have 
\begin{eqnarray}
\Lambda(t)(\deg^{t}X_{i},\deg^{t}X_{j}) & = & \Lambda(t')(\deg^{t'}X_{i},\deg^{t'}X_{j})\label{eq:lambda_XX}\\
\Lambda(t)(\deg^{t}X_{i},\deg^{t}S_{k}) & = & \Lambda(t')(\deg^{t'}X_{i},\deg^{t'}S_{k})\label{eq:lambda_XS}\\
\Lambda(t)(\deg^{t}S_{l},\deg^{t}S_{k}) & = & \Lambda(t')(\deg^{t'}S_{l},\deg^{t'}S_{k}).\label{eq:lambda_SS}
\end{eqnarray}

Notice that we have $\deg^{t}S_{l}=-e_{l}+\sum_{s}[-b_{sl}]_{+}e_{s}$,
where all $e_{s}$ appearing have $s\neq n$. Therefore, we deduce
that $\deg^{t'}S_{l}=\deg^{t}S_{l}$, $\forall l<n$, by the tropical
transformation of $g$-vectors, cf. \cite[Section 3.2]{Qin15}\cite{FockGoncharov03}\cite[(7.18)]{FominZelevinsky07}.
The first two equations simply follows from the mutation rule from
$\Lambda(t)$ to $\Lambda(t')$. It remains to check \eqref{eq:lambda_SS}.
By using \eqref{eq:lambda_XS}, we obtain

\begin{eqnarray*}
 &  & \Lambda(t)(\deg^{t}S_{l},\deg^{t}S_{k})\\
 & = & \Lambda(t)(-\deg^{t}X_{l}+\sum_{s}[-b_{sl}]_{+}\deg^{t}X_{s},\deg^{t}S_{k})\\
 & = & -\Lambda(t)(\deg^{t}X_{l},\deg^{t}S_{k})+\sum_{s}[b_{-sl}]_{+}\Lambda(t)(\deg^{t}X_{s},\deg^{t}S_{k})\\
 & = & -\Lambda(t')(\deg^{t'}X_{l},\deg^{t'}S_{k})+\sum_{s}[b_{-sl}]_{+}\Lambda(t')(\deg^{t'}X_{s},\deg^{t'}S_{k})\\
 & = & \Lambda(t')(\deg^{t'}S_{l},\deg^{t'}S_{k}).
\end{eqnarray*}

\end{proof}

The following statement is the main result of \cite{KimuraQin14}
accompanied with the coefficient correction technique in \cite{Qin12}.

\begin{Thm}[\cite{KimuraQin14}\cite{Qin12}]\label{thm:ayclic_triangular_basis}

If the principal part $B(t)$ of a seed $t$ is acyclic and skew-symmetric,
then the triangular basis $\can^{t}$ for $t$ exists. Moreover, it
contains all the quantum cluster monomials.

\end{Thm}

\begin{proof}

When we choose the special coefficient pattern $B^{c}(t)$ to be $z$-pattern
as in \cite{KimuraQin14}, the quantum cluster algebra is isomorphic
to a subalgebra of a quantized enveloping algebra \cite{GeissLeclercSchroeer11}.
Under this identification, $X_{i}(t)$, $I_{k}(t)$ are the factors
of the dual PBW basis element, and the triangular basis $\can^{t}$
is just the restriction of the dual canonical basis on this subalgebra
(and localized at the coefficients $(X_{n+1},\cdots,X_{m})$). By
\cite{KimuraQin14}, this basis contains all the quantum cluster monomials.

By the correction technique in \cite{Qin12}, we can change the coefficient
pattern $B^{c}(t)$ and $\Lambda(t)$ while keeping the claim true.

\end{proof}

The following statement is implied by the general result in \cite[Theorem 9.4.1]{Qin15}.
We sketch a much easier proof for this special case.

\begin{Thm}\label{thm:common_triangular_basis}

Let $t$ and $t'$ be two seeds such that $t'=\mu_{k}t$ for some
$1\leq k\leq n$ and $B(t)$, $B(t')$ are acyclic and skew-symmetric.
Then the quantum cluster algebra has a basis $\can$ which is the
triangular basis for both $t$ and $t'$.

\end{Thm}

\begin{proof}

Because $t$ and $t'$ are acyclic, by Theorem \ref{thm:ayclic_triangular_basis},
we know that the triangular bases $\can^{t}$ and $\can^{t'}$ for
$t$ and $t'$ exist. Moreover, the quantum cluster monomials $X_{k}'^{d}=X_{k}(t')^{d}$,
$I_{k}'^{d}=I_{k}(t')^{d}$ belong to $\can^{t}$, where $d\in\N$.
Therefore, $X_{k}'^{d}$ and $I_{k}'^{d}$ have $(\prec_{t},\Bm)$-unitriangular
decomposition in the injective pointed set $\inj^{t}$. These are
the only new factors of elements in $\inj^{t'}$ which are not factors
of elements in $\inj^{t}$. 

Easy calculation shows that elements in $\inj^{t'}$ remain pointed
in $\cT(t)$, cf. \cite[Lemma 5.3.2]{Qin15}. Substituting their new
factors $X_{k}'^{d}$ and $I_{k}'^{d}$ by the decomposition in $\inj^{t}$,
we deduce that $\inj^{t'}$ is $(\prec_{t},\Bm)$-unitriangular to
$\inj^{t}$ by Lemma \ref{lem:substitution}.

Also, notice that $\can^{t'}$ is $(\prec_{t'},\Bm)$-unitriangular
to $\inj^{t'}$ and $\inj^{t}$ is $(\prec_{t},\Bm)$-unitriangular
to $\can^{t}$. Composing these three transitions, we obtain that
any $S'\in\can^{t'}$ is a finite combination of elements $S$, $S_{i}$
in $\can^{t}$:

\begin{align*}
S' & =S+\sum_{i}a_{i}S_{i},
\end{align*}

with coefficient $a_{i}\in\Bm$.

Now by the bar-invariance of $\can^{t}$ and $\can^{t'}$, we must
have $a_{i}=0$ and $S'=S$. It follows that the two triangular bases
$\can^{t}$ and $\can^{t'}$ are the same.

\end{proof}

\subsection{Proof of the main result}

For any chosen $1\leq j\leq n$, let $t[j^{-1}]$ denote the seed
obtained from $t$ by deleting the $j$-th column in the matrix $\tB(t)$.
This operation is called freezing the vertex $j$. We have the corresponding
quantum cluster algebra $\qClAlg(t[j^{-1}])$. Observe that the normalization
$[\ ]^{t[j^{-1}]}=[\ ]^{t}$ because $\Lambda(t[j^{-1}])=\Lambda(t)$
by construction. Moreover, the partial order $\prec_{t[n^{-1}]}$
implies $\prec_{t}$ by definition. We can define similarly, for $f\in\Z^{\{j\}\cup[n+1,m]},u,v\in\N^{[1,n]-\{j\}}$,
where $u_{k}v_{k}=0$ for any $k$: 
\begin{align*}
 & \inj^{t[j^{-1}]}(f,u,v)\\
 & =[\prod_{n+1\leq i\leq m}X_{i}^{f_{i}}*X_{j}^{f_{j}}*\prod_{1\leq k\leq n,k\neq j}X_{k}^{u_{k}}*\prod_{1\leq k\leq n,k\neq j}I_{k}(t[j^{-1}])^{v_{k}}]^{t[j^{-1}]}.
\end{align*}

We want to compare this new injective pointed set $\inj^{t[j^{-1}]}$
with the old one $\inj^{t}$. One has to pay attention to the possible
localization at $X_{j}$ in the seed $t[j^{-1}]$.

Assume the vertex $n$ to be $\triangleleft$-maximal, namely, a source point,
then $I_{k}(t[n^{-1}])=I_{k}(t)$ for all $1\leq k< n$,
cf. Remark \ref{rem:acyclic_injectives}, and,  moreover, $(\deg Y_{i})_{n}=b_{ni}\geq 0$
$\forall 1\leq i\leq n$. It follows that the Laurent monomials of $I_{k}(t)$, $\forall k\neq n$, have non-negative degrees in $X_{n}$.

Notice that, for a source point $n$, if $f_{n}\geq0$,
then $\inj^{t[n^{-1}]}(f,u,v)\in\inj^{t}$.

\begin{Lem}\label{lem:restricted_basis}

Assume that $n$ is a source point and a pointed element $Z\in\qClAlg(t[n^{-1}])$
has a finite combination of 
\begin{align*}
Z & =\sum_{s}a_{s}\inj^{t[n^{-1}]}(f^{(s)},u^{(s)},v^{(s)}).
\end{align*}

If $(\deg Z)_{n}\geq0$, then we have $f_{n}^{(s)}\geq0$ whenever
$a_{s}\neq0$. Consequently, all $\inj^{t[n^{-1}]}(f^{(s)},u^{(s)},v^{(s)})$
appearing in the combination are contained in $\inj^{t}$.

\end{Lem}

\begin{proof}
Recall that $\inj^{t[n^{-1}]}$ is a linearly independent family of
pointed elements with distinguished leading degrees. By Lemma \ref{lem:has_triangular_order}(i), the given decomposition of $Z$ is $\prec_t$-unitriangular with a unique  leading term $\inj^{t[n^{-1}]}(f^{(0)},u^{(0)},v^{(0)})$ whose leading degree equals $\deg Z$. So the leading degrees of all $\inj^{t[n^{-1}]}(f^{(s)},u^{(s)},v^{(s)})$ appearing are $\prec_t$-inferior or equal to $\deg Z$. Since $(\deg Z)_n\geq 0$ and $(\deg Y_i)_n\geq 0$, $\forall 1\leq i\leq n$, they are all non-negative in the $n$-th components.

Notice that $\pr_n \deg I_{k}(t)=-e_k$ by definition and, in particular, the leading degree $\deg I_{k}(t)$, $\forall k<n$, vanishes in the $n$-th components.
It follows that $\deg \inj^{t[n^{-1}]}(f^{(s)},u^{(s)},v^{(s)})$ has non-negative $n$-th component if and only if $f^{(s)}_n\geq 0$. The claim follows.
\end{proof}

\begin{proof}[Proof of Theorem \ref{thm:acyclic}]

We prove the claim by induction on the rank $n$ of $\tB(t)$. The
cases $n=0$ are trivial.

Up to relabeling vertices, let us assume that $n$ is a source point
in $t$. Denote $t'=\mu_{n}t$.

It suffices to show that every $\overline{E}_{a}$, $a\in\Z^{m}$,
is $(\prec_{t},\Bm)$-triangular to $\can^{t}$. If so, combined with
Lemma \ref{lem:change_BZ_order}, we obtain that every bar-invariant
element $C_{a}$ is $(\prec_{t},\Bm)$-triangular to $\can^{t}$ and,
consequently, must belong to $\can^{t}$. It follows that the two
bases $\can^{t}$ and $C$ must agree.

(i) Assume $a_{n}\geq0$. Consider the seed $t[n^{-1}]$ obtained
by freezing the vertex $n$ in $t$. It is acyclic whose matrix $\tB(t[n^{-1}])$
has rank $n-1$. By induction hypothesis, its triangular basis $\can^{t[n^{-1}]}$
agrees with its BZ-basis $C^{t[n^{-1}]}$. Notice that the corresponding
standard monomial $E_{a}$ is also a standard monomials for seed $t[n^{-1}]$.
Therefore, $\overline{E}_{a}$ admits a finite decomposition in $C^{t[n^{-1}]}=\can^{t[n^{-1}]}$
with one term of coefficient $1$ and other terms of coefficient in
$\Bm$. Recall that $\can^{t[n^{-1}]}$ is $\prec_{t[n^{-1}]}$-unitriangular
to $\inj^{t[n^{-1}]}$. Composing these two transitions, we see that
$\overline{E}_{a}$ has a finite decomposition in $\inj^{t[n^{-1}]}$
with one term of coefficient $1$ and others of coefficient in $\Bm$.
Further notice that $(\deg\overline{E}_{a})_{n}\geq0$, by Lemma \ref{lem:restricted_basis},
the decomposition terms appearing belong to $\inj^{t}$. By Lemma
\ref{lem:has_triangular_order}, $\overline{E}_{a}$ is $(\prec_{t},\Bm)$-unitriangular
to $\inj^{t}$, and consequently $(\prec_{t},\Bm)$-unitriangular
to $\can^{t}$. 

(ii) When $a_{n}<0$, let us rewrite $\overline{E}_{a}$ as $[S_{n}^{[-a_{n}]_{+}}*\overline{E}_{a_{\hat{n}}}]^{t}$,
where $a_{\hat{n}}$ denote the vector obtained from $a$ by setting
the $n$-th component to $0$. Notice that $\overline{E}_{a}$ is
also pointed in $t'$ by Lemma \ref{lem:keep_pointed}, namely, $\overline{E}_{a}=[S_{n}^{[-a_{n}]_{+}}*\overline{E}_{a_{\hat{n}}}]^{t'}$.
For the seed $t'$, we freeze the vertex $n$ and repeat the argument
in $(i)$, it follows that $\overline{E}_{a_{\hat{n}}}$ is $(\prec_{t'},\Bm)$-unitriangular
to the triangular basis $\can^{t'}$ of the seed $t'$. Notice that
$S_{n}$ is the $n$-th cluster variable in the seed $t'$. By Lemma
\ref{lem:substitution}, we obtain that $\overline{E}_{a}$ is $(\prec_{t'},\Bm)$-unitriangular
to the triangular basis $\can^{t'}$ of the seed $t'$. Because $\can^{t}=\can^{t'}$
by Theorem \ref{thm:common_triangular_basis}, $\overline{E}_{a}$
is $(\prec_{t},\Bm)$-unitriangular to $\can^{t}$ by Lemma \ref{lem:has_triangular_order}.

\end{proof}

\subsection{Bipartite skew-symmetrizable case}

We say the seed $t$ has a bipartite orientation (we say $t$ is bipartite
for short), if we have $\{1,\cdots,n\}=V_{0}\sqcup V_{1}$, such that
all the vertices in $V_{0}$ are source points and those in $V_{1}$
are sink points. 

Assume that $t$ is bipartite. Let we denote by $t'$ the seed obtained
from $t$ by mutating at all the vertices in $V_{1}$, namely,

\begin{align*}
\mu_{V_{1}} & =\prod_{k\in V_{1}}\mu_{k}\\
t' & =\mu_{V_{1}}t.
\end{align*}

Notice that the mutations $\mu_{k}$, $k\in V_{1}$, commute with
each other.

The following lemma follows from the definitions of the corresponding
cluster variables, cf. Figure \ref{fig:knitting} for identification
of cluster variables, where $i\in V_{0}$, $j\in V_{1}$, the graph
are constructed via the knitting algorithm, cf. \cite{Keller08Note}.

\begin{Lem}

\label{lem:comparison} We have, for any $1\leq i,j\leq n$, 
\begin{align}
X_{i}(t')=X_{i}(t),\ i\in V_{0},\\
X_{j}(t')=I_{j}(t),\ j\in V_{1},\\
S_{i}(t')=I_{i}(t),\ i\in V_{0},\\
S_{j}(t')=X_{j}(t),\ j\in V_{1}.
\end{align}

\end{Lem}

\begin{figure}[htb!]  \centering \beginpgfgraphicnamed{fig:knitting}   \begin{tikzpicture}     \node [ draw] (v1) at (1,-3) {$X_i(t)$};   \node  [ draw] (v2) at (2,-1.5) {$X_j(t)$}; \node  [draw] (v3) at (0,-1.5) {$I_j(t)$};   \node [draw] (v4) at (-1,-3) {$I_i(t)$};   
  \draw[-triangle 60] (v1) edge (v2);   \draw[-triangle 60] (v3) edge (v1);  \draw[-triangle 60] (v4) edge (v3);
\node [ draw] (v11) at (6,-3) {$X_i(t')$};   \node  [ draw] (v12) at (7,-1.5) {$S_j(t')$}; \node  [draw] (v13) at (5,-1.5) {$X_j(t')$};  \node [ draw] (v14) at (4,-3) {$S_i(t')$};    
  \draw[-triangle 60] (v11) edge (v12);   \draw[-triangle 60] (v13) edge (v11);  \draw[-triangle 60] (v14) edge (v13);
\end{tikzpicture} \endpgfgraphicnamed \caption{Part of knitting graphs for the seeds $t$ and $t'$.} \label{fig:knitting} \end{figure}
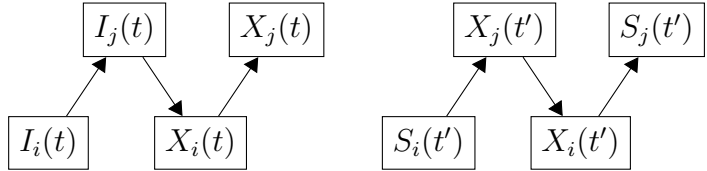

It follows from Lemma \ref{lem:comparison} that those $S(t')$, $i\in V_{0}$
$q$-commute with each other, and $S_{j}(t')$, $j\in V_{1}$, $q$-commute
with each other.

Notice that $t'$ is still bipartite with the vertices in $V_{0}$
being sink points and the vertices in $V_{1}$ being source points.

\begin{Lem}\label{lem:variable_commute}

(1) For any $1\leq k\neq j\leq n$, such that $j\in V_{1}$, $X_{k}(t)$
and $I_{j}(t)$ $q$-commute.

(2) For any $1\leq i\neq k\leq n$, such that $i\in V_{0}$, $X_{i}(t)$
and $I_{k}(t)$ $q$-commute

\end{Lem}

\begin{proof}

(1) $X_{k}(t)$ and $I_{j}(t)$ are quantum cluster variables in the
same seed $\mu_{j}t$.

(2) By (1), it remains to check the case $i,k\in V_{0}$. Notice that
$V_{0}$ consist of sink points in $t'=\mu_{V_{1}}t$. $X_{i}(t)$
and $I_{k}(t)$ are quantum cluster variables in the same seed $\mu_{k}t'$.

\end{proof}

\begin{Lem}\label{lem:keep_pointed_bipartite}

The pointed element $\overline{E}_{a}$ defined in $t'$ remains pointed
in $t=\mu_{V_{1}}t'$.

\end{Lem}

\begin{proof}

The vertices in $V_{1}$ are source points in $t'$ which are not
connected by arrows. We simply repeat the proof of Lemma \ref{lem:keep_pointed}.

\end{proof}

\begin{Thm}\label{thm:bipartite}

For bipartite $t$, the Berenstein-Zelevinsky's triangular basis $C$
is also the triangular basis $\can^{t}$.

\end{Thm}

\begin{proof}

Notice that, in the seed $t'$, the vertices in $V_{1}$ are source
points and $\triangleleft$-superior than those in $V_{0}$. Using
Lemma \ref{lem:variable_commute}(ii), we have, for any $a\in\Z^{m}$, 

\begin{align}
\overline{E}_{a} & =[\prod_{j\in V_{1}}S_{j}(t')^{[-a_{j}]_{+}}*\prod_{i\in V_{0}}S_{i}(t')^{[-a_{i}]_{+}}*\prod_{j\in V_{1}}X_{j}(t')^{[a_{j}]_{+}}*\prod_{i\in V_{0}}X_{i}(t')^{[a_{i}]_{+}}*\prod_{n+1\leq j\leq m}X_{j}(t')^{a_{j}}]^{t'}\nonumber \\
 & =[\prod_{j\in V_{1}}X_{j}(t)^{[-a_{j}]_{+}}*\prod_{i\in V_{0}}I_{i}(t)^{[-a_{i}]_{+}}*\prod_{j\in V_{1}}I_{j}(t)^{[a_{j}]_{+}}*\prod_{i\in V_{0}}X_{i}(t)^{[a_{i}]_{+}}*\prod_{n+1\leq j\leq m}X_{j}(t')^{a_{j}}]^{t'}.\nonumber \\
 & =[\prod_{j\in V_{1}}X_{j}(t)^{[-a_{j}]_{+}}*\prod_{i\in V_{0}}X_{i}(t)^{[a_{i}]_{+}}*\prod_{i\in V_{0}}I_{i}(t)^{[-a_{i}]_{+}}*\prod_{j\in V_{1}}I_{j}(t)^{[a_{j}]_{+}}*\prod_{n+1\leq j\leq m}X_{j}(t')^{a_{j}}]^{t'}\label{eq:rewrite_monomial}
\end{align}

By Lemma \ref{lem:keep_pointed_bipartite}, $\overline{E}_{a}$ remains
to be pointed in $t$. Then \eqref{eq:rewrite_monomial} tells us
that it belongs to the injective pointed set $\inj^{t}$. All elements
of $\inj^{t}$ take this form. So we see the BZ-basis $C$ verifies
the conditions (i)(ii)(iv) in Definition \ref{def:triangular_basis}.
A closer examination tells us that the condition (iii) in Definition
\ref{def:triangular_basis} is also verified by the basis $C$, cf.
\cite{BerensteinZelevinsky2012}. So $C$ is the triangular basis
$\can^{t}$ for the seed $t$.

\end{proof}

\begin{Eg}[Kronecker quiver type]\label{eg:Kronecker}

Let us look at the quantum cluster algebra with the seed $t$ given
by $\tB=\left(\begin{array}{cc}
0 & 2\\
-2 & 0
\end{array}\right)$ and $\Lambda=\left(\begin{array}{cc}
0 & 1\\
-1 & 0
\end{array}\right)$. We have the set of source points $V_{0}=\{1\}$ and the set of sink
points $V_{1}=\{2\}$.

Its seed $t'=\mu_{V_{1}}t$ has the matrices $\tB=\left(\begin{array}{cc}
0 & -2\\
2 & 0
\end{array}\right)$ and $\Lambda=\left(\begin{array}{cc}
0 & -1\\
1 & 0
\end{array}\right)$. The vertex $2$ is the source point in $t'$. It is easy to compute
that

\begin{eqnarray*}
S_{1}(t') & = & X(t')^{-e_{1}}+X(t')^{-e_{1}+2e_{2}}\\
S_{2}(t') & = & X(t')^{-e_{2}+2e_{1}}+X(t')^{-e_{2}}\\
Y_{1}(t') & = & X(t')^{2e_{2}}\\
Y_{2}(t') & = & X(t')^{-2e_{1}}.
\end{eqnarray*}

By \cite[(6.4)]{BerensteinZelevinsky2012}\cite{ding2012bases}, we
have the following bar-invariant pointed element $X_{\delta}$ in
the BZ-basis $C$, given by

\begin{align*}
X_{\delta} & =q^{\Hf}S_{1}(t')*S_{2}(t')-q^{\frac{3}{2}}X_{2}(t')*X_{1}(t')\\
 & =X(t')^{e_{1}-e_{2}}\cdot(1+Y_{2}(t')+Y_{1}(t')Y_{2}(t'))\\
 & =X(t')^{e_{1}-e_{2}}+X(t')^{-e_{1}-e_{2}}+X(t')^{e_{2}-e_{1}}.
\end{align*}

Taking the bar-involution, we obtain

\begin{eqnarray*}
X_{\delta} & = & q^{-\Hf}S_{2}(t')*S_{1}(t')-q^{-\frac{3}{2}}X_{1}(t')*X_{2}(t')\\
 & = & [S_{2}(t')*S_{1}(t')]^{t'}-q^{-2}[X_{1}(t')*X_{2}(t')]^{t'}.
\end{eqnarray*}

We have 
\begin{eqnarray*}
S_{2}(t') & = & X_{2}(t)\\
S_{1}(t') & = & I_{1}(t)\\
 & = & X(t)^{-e_{1}}(1+Y_{1}(t)+(q+q^{-1})Y_{1}(t)Y_{2}(t)+Y_{1}(t)Y_{2}(t)^{2})\\
X_{2}(t') & = & I_{2}(t)\\
 & = & X(t)^{-e_{2}}(1+Y_{2}(t))\\
X_{1}(t') & = & X_{1}(t)
\end{eqnarray*}

Then $X_{\delta}$ can be rewritten as

\begin{align*}
X_{\delta} & =[X_{2}(t)*I_{1}(t)]^{t}-q^{-2}[X_{1}(t)*I_{2}(t)]^{t}\\
 & =X(t)^{-e_{1}+e_{2}}(1+Y_{1}(t)+(1+q^{-2})Y_{1}(t)Y_{2}(t)+q^{-2}Y_{1}(t)Y_{2}(t)^{2})\\
 & \qquad-q^{-2}X^{e_{1}-e_{2}}(1+Y_{2}(t))\\
 & =X(t)^{-e_{1}+e_{2}}(1+Y_{1}(t)+Y_{1}(t)Y_{2}(t))\\
 & =X(t)^{-e_{1}+e_{2}}+X(t)^{-e_{1}-e_{2}}+X(t)^{e_{1}-e_{2}}.
\end{align*}

Notice that the normalization factors do not change:

\begin{align*}
\Lambda(t)(\deg^{t}X_{2}(t),\deg^{t}I_{1}(t)) & =\Lambda(t)(e_{2},-e_{1})= & 1 & =\Lambda(t')(\deg^{t'}S_{2}(t'),\deg^{t'}S_{1}(t'))\\
\Lambda(t)(\deg^{t}X_{1}(t),\deg^{t}I_{2}(t)) & =\Lambda(t)(e_{1},-e_{2})= & -1 & =\Lambda(t')(\deg^{t'}X_{1}(t'),\deg^{t}X_{2}(t')).
\end{align*}

Therefore, the pointed element $X_{\delta}$ is $(\prec_{t},\Bm)$-unitriangular
to the injective pointed set $\inj^{t}$, and consequently $(\prec_{t},\Bm)$-unitriangular
to the triangular basis $\can^{t}$. It follows from its bar-invariance
that $X_{\delta}$ belongs to the triangular basis $\can^{t}$. 

\end{Eg}


\bibliographystyle{halpha}
\bibliography{referenceEprint}

\def\cprime{$'$}
\begin{thebibliography}{KKKO15}

\bibitem[BZ05]{BerensteinZelevinsky05}
Arkady Berenstein and Andrei Zelevinsky.
\newblock Quantum cluster algebras.
\newblock {\em Adv. Math.}, 195(2):405--455, 2005, math/0404446v2.

\bibitem[BZ12]{BerensteinZelevinsky12}
Arkady Berenstein and Andrei Zelevinsky.
\newblock Triangular bases in quantum cluster algebras.
\newblock 2012, 1206.3586.

\bibitem[BZ14]{BerensteinZelevinsky2012}
Arkady Berenstein and Andrei Zelevinsky.
\newblock Triangular bases in quantum cluster algebras.
\newblock {\em International Mathematics Research Notices}, 2014(6):1651--1688,
  2014, 1206.3586.

\bibitem[DWZ08]{DerksenWeymanZelevinsky08}
Harm Derksen, Jerzy Weyman, and Andrei Zelevinsky.
\newblock Quivers with potentials and their representations {I}: {Mutations}.
\newblock {\em Selecta Mathematica}, 14:59--119, 2008, 0704.0649v4.

\bibitem[DWZ10]{DerksenWeymanZelevinsky09}
Harm Derksen, Jerzy Weyman, and Andrei Zelevinsky.
\newblock Quivers with potentials and their representations {II}: {Applications
  to cluster algebras}.
\newblock {\em J. Amer. Math. Soc.}, 23(3):749--790, 2010, 0904.0676v2.

\bibitem[DX12]{ding2012bases}
Ming Ding and Fan Xu.
\newblock Bases of the quantum cluster algebra of the kronecker quiver.
\newblock {\em Acta Mathematica Sinica, English Series}, 28(6):1169--1178,
  2012.

\bibitem[FG09]{FockGoncharov03}
V.~V. Fock and A.~B. Goncharov.
\newblock {Cluster ensembles, quantization and the dilogarithm}.
\newblock {\em Ann. Sci. \'Ecole Norm. Sup. (4)}, 42(6):865--930, 2009,
  math.AG/0311245.

\bibitem[FZ02]{fomin2002cluster}
Sergey Fomin and Andrei Zelevinsky.
\newblock Cluster algebras {I}: foundations.
\newblock {\em Journal of the American Mathematical Society}, 15(2):497--529,
  2002.

\bibitem[FZ07]{FominZelevinsky07}
Sergey Fomin and Andrei Zelevinsky.
\newblock Cluster algebras {IV}: Coefficients.
\newblock {\em Compositio Mathematica}, 143:112--164, 2007, math/0602259v3.

\bibitem[GHKK14]{gross2014canonical}
Mark Gross, Paul Hacking, Sean Keel, and Maxim Kontsevich.
\newblock Canonical bases for cluster algebras.
\newblock 2014, 1411.1394.

\bibitem[GLS11]{GeissLeclercSchroeer10}
Christof Gei\ss, Bernard Leclerc, and Jan Schr{\"o}er.
\newblock Kac-{M}oody groups and cluster algebras.
\newblock {\em Advances in Mathematics}, 228(1):329--433, 2011, 1001.3545v2.

\bibitem[GLS12]{GeissLeclercSchroeer10b}
Christof Gei\ss, Bernard Leclerc, and Jan Schr{\"o}er.
\newblock Generic bases for cluster algebras and the {C}hamber {A}nsatz.
\newblock {\em J. Amer. Math. Soc.}, 25(1):21--76, 2012, 1004.2781v3.

\bibitem[GLS13]{GeissLeclercSchroeer11}
Christof Gei\ss, Bernard Leclerc, and Jan Schr{\"o}er.
\newblock Cluster structures on quantum coordinate rings.
\newblock {\em Selecta Mathematica}, 19(2):337--397, 2013, 1104.0531.

\bibitem[HL10]{HernandezLeclerc09}
David Hernandez and Bernard Leclerc.
\newblock Cluster algebras and quantum affine algebras.
\newblock {\em Duke Math. J.}, 154(2):265--341, 2010, 0903.1452.

\bibitem[Kas90]{Kashiwara90}
Masaki Kashiwara.
\newblock Bases cristallines.
\newblock {\em C. R. Acad. Sci. Paris S\'er. I Math.}, 311(6):277--280, 1990.

\bibitem[Kel08]{Keller08Note}
Bernhard Keller.
\newblock Cluster algebras, quiver representations and triangulated categories.
\newblock 2008, 0807.1960v11.

\bibitem[KKKO15]{KKKO15}
S.-J. {Kang}, M.~{Kashiwara}, M.~{Kim}, and S.-j. {Oh}.
\newblock {Monoidal categorification of cluster algebras II}.
\newblock {\em ArXiv e-prints}, 2015, 1502.06714.

\bibitem[KQ14]{KimuraQin14}
Yoshiyuki Kimura and Fan Qin.
\newblock Graded quiver varieties, quantum cluster algebras and dual canonical
  basis.
\newblock {\em Advances in Mathematics}, 262:261--312, 2014, 1205.2066.

\bibitem[LLRZ14]{lee2014greedyPNAS}
Kyungyong Lee, Li~Li, Dylan Rupel, and Andrei Zelevinsky.
\newblock Greedy bases in rank 2 quantum cluster algebras.
\newblock {\em Proceedings of the National Academy of Sciences},
  111(27):9712--9716, 2014.

\bibitem[LLZ14]{lee2014greedy}
Kyungyong Lee, Li~Li, and Andrei Zelevinsky.
\newblock Greedy elements in rank 2 cluster algebras.
\newblock {\em Selecta Mathematica}, 20(1):57--82, 2014.

\bibitem[Lus90]{Lusztig90}
G.~Lusztig.
\newblock Canonical bases arising from quantized enveloping algebras.
\newblock {\em J. Amer. Math. Soc.}, 3(2):447--498, 1990.

\bibitem[MSW13]{musiker2013bases}
Gregg Musiker, Ralf Schiffler, and Lauren Williams.
\newblock Bases for cluster algebras from surfaces.
\newblock {\em Compositio Mathematica}, 149(02):217--263, 2013.

\bibitem[Nak11]{Nakajima09}
Hiraku Nakajima.
\newblock Quiver varieties and cluster algebras.
\newblock {\em Kyoto J. Math.}, 51(1):71--126, 2011, 0905.0002v5.

\bibitem[Pla11]{Plamondon10a}
Pierre-Guy Plamondon.
\newblock Cluster characters for cluster categories with
  in\-fi\-ni\-te-\-di\-men\-sional morphism spaces.
\newblock {\em Adv. in Math.}, 227(1):1--39, 2011, 1002.4956v2.

\bibitem[Qin14]{Qin12}
Fan Qin.
\newblock t-analog of q-characters, bases of quantum cluster algebras, and a
  correction technique.
\newblock {\em International Mathematics Research Notices},
  2014(22):6175--6232, 2014, 1207.6604.

\bibitem[Qin15]{Qin15}
Fan Qin.
\newblock Triangular bases in quantum cluster algebras and monoidal
  categorification conjectures.
\newblock 2015, 1501.04085.

\bibitem[Thu14]{thurston2014positive}
Dylan~Paul Thurston.
\newblock Positive basis for surface skein algebras.
\newblock {\em Proceedings of the National Academy of Sciences},
  111(27):9725--9732, 2014.

\end{thebibliography}

\end{document}